\documentclass[12pt,reqno]{amsart}
\usepackage{enumerate, latexsym, amsmath, amsfonts, amssymb, amsthm, graphicx, color}
\def\pmod #1{\ ({\rm{mod}}\ #1)}
\def\Z{\Bbb Z}
\def\N{\Bbb N}

\def\Q{\Bbb Q}

\def\bg{\bigg}
\def\({\bg(}
\def\){\bg)}

\def\gen{{\rm gen}}

\def\ve{\varepsilon}

\theoremstyle{plain}
\newtheorem{theorem}{Theorem}

\newtheorem{lemma}{Lemma}

\theoremstyle{definition}

\theoremstyle{remark}
\newtheorem{remark}{Remark}

\makeatletter
\@namedef{subjclassname@2010}{%
  \textup{2010} Mathematics Subject Classification}
\makeatother
 \vspace{4mm}

\begin{document}
 \baselineskip=17pt
\hbox{}
\medskip
\title[Almost universal mixed sums of squares and polygonal numbers ]
{Almost universal mixed sums of squares and polygonal numbers}
\date{}
\author[Hai-Liang Wu] {Hai-Liang Wu}

\thanks{
\newline \indent Supported by the National Natural Science
Foundation of China (Grant No. 11571162).}

\address {(Hai-Liang Wu)  Department of Mathematics, Nanjing
University, Nanjing 210093, People's Republic of China}
\email{\tt whl.math@smail.nju.edu.cn}

\begin{abstract}
For each integer $m\ge3$, let $P_m(x)$ denote the generalized $m$-gonal number $\frac{(m-2)x^2-(m-4)x}{2}$ with $x\in\Z$. Given positive integers $a,b,c,k$ and an odd prime number $p$ with $p\nmid c$, we employ the theory of ternary quadratic forms to determine completely when the mixed sum $ax^2+by^2+cP_{p^k+2}(z)$ represents all but finitely many positive integers.

\end{abstract}

\maketitle

\section{Introduction}
\setcounter{lemma}{0}
\setcounter{theorem}{0}
\setcounter{corollary}{0}
\setcounter{remark}{0}
\setcounter{equation}{0}
\setcounter{conjecture}{0}

For a natural number $m\ge3$, the generalized $m$-gonal number is given by $P_m(x)=\frac{(m-2)x^2-(m-4)x}{2}$
where $x\in\Z$. In 1796 Gauss proved Fermat's assertion that each positive integer can
be expressed as the sum of three triangular numbers (corresponding to m = 3). In 1862 Liouville (cf. Berndt \cite[p.82]{Be}) determined all weighted ternary sums of triangular numbers that represent all natural numbers.
In $2007$, Sun \cite{S07} investigated the mixed sums of squares and triangular numbers. In this direction, Kane and Sun \cite{KS} obtained a classification of almost universal weighted sums of triangular numbers and more generally weighted mixed ternary sums of triangular
and square numbers (a quadratic polynomial is said to be almost universal, if it represents all but finitely many positive integers over $\Z$), this classification was later completed by Chan and Oh \cite{WKCOH}
and Chan and Haensch \cite{WKCANNA}. A. Haensch \cite{ANNA} investigated the almost universal ternary quadratic polynomials with odd prime power conductor.
Recently, Sun \cite{S17} showed that there are totally 12082 possible tuples $(a,b,c,d,e,f)$
with $a\ge c\ge e\ge1$, $b\equiv a\pmod2$ and $|b|\le a$, $d\equiv c\pmod2$ and $|d|\le c$, $f\equiv e\pmod2$
and $|f|\le e$, such that the sum
$$ \frac{x(ax+b)}{2}+\frac{y(cy+d)}{2}+\frac{z(ez+f)}{2}.$$ represents all natural numbers over $\Z$.

Motivated by these works, we shall give a complete characterization of all the triples of
positive integers $(a,b,c)$ for which the ternary sums $ax^2+by^2+cP_{p^k+2}(z)$ are almost universal over $\Z$, where $k>0$ and $p$ is an odd prime not dividing $c$.

Now, we state our main results in this paper. Throughout this paper, without loss of generality, we may assume that $\nu_2(a)\ge\nu_2(b)$. For convenience, the squarefree part of an integer $m$ is denoted by $\mathcal {SF}(m)$ and the odd part of $m$ is denoted by $m'$.

\begin{theorem}\label{Thm3.1}
Let $a,b,c\in\Z^+$ with $gcd(a,b,c)=1$, $\nu_p(a)\equiv\nu_p(b)\pmod 2$ and $\nu_2(a)\ge\nu_2(b)\ge2$.
Suppose that both $(1)$ and $(2)$ in Lemma \ref{Lem2.1} hold. Then $f_{a,b,c,p^k}$ is not almost universal
if and only if all of the following are satisfied:

{\rm (1)} Each prime divisor of $\mathcal{SF}(a'b'c')$ is congruent to $1$ modulo $4$ if
$\nu_2(a)\equiv\nu_2(b)\pmod2$, and is congruent to $1,3$ modulo $8$ if $\nu_2(a)\not\equiv\nu_2(b)\pmod2$.

{\rm (2)} $\nu_p(a)\equiv\nu_p(b)\equiv k\pmod 2$.

{\rm (3)} $a'\equiv b'\pmod 8$, and
$$\begin{cases}p^kb'c'\equiv1,3\pmod8&\mbox{if}\ \nu_2(a)\not\equiv\nu_2(b)\pmod2,\\p^kb'c'\equiv1\pmod4&\mbox{if}\ \nu_2(a)\equiv\nu_2(b)\pmod2.\end{cases}$$

{\rm (4)} $\mathcal{SF}(a'b'c')c^{-1}$ is a quadratic residue modulo $p^k$, and
$ax^2+by^2+cP_{p^k+2}(z)=(\mathcal{SF}(a'b'c')-c(p^k-2)^2)/(8p^k)$ has no integral solutions.
\end{theorem}

\begin{theorem}\label{Thm3.2}
Let $a,b,c\in\Z^+$ with $gcd(a,b,c)=1$, $\nu_p(a)\equiv\nu_p(b)\pmod 2$ and $\nu_2(a)\ge\nu_2(b)=1$.
Suppose that both $(1)$ and $(2)$ in Lemma \ref{Lem2.1} hold. Then $f_{a,b,c,p^k}$ is not almost universal
if and only if all of the following are satisfied:

{\rm (1)} Each prime divisor of $\mathcal{SF}(a'b'c')$ is congruent to $1$ modulo $4$ if
$\nu_2(a)\equiv\nu_2(b)\pmod2$, and is congruent to $1,3$ modulo $8$ if $\nu_2(a)\not\equiv\nu_2(b)\pmod2$.

{\rm (2)} $\nu_p(a)\equiv\nu_p(b)\equiv k\pmod 2$.

{\rm (3)} $a'\equiv b'\pmod 8$, $\nu_2(a)\equiv\nu_2(b)\pmod 2$ and $p^kb'c'\equiv 1\pmod 4$.

{\rm (4)} $\mathcal{SF}(a'b'c')c^{-1}$ is a quadratic residue modulo $p^k$, and
$ax^2+by^2+cP_{p^k+2}(z)=(\mathcal{SF}(a'b'c')-c(p^k-2)^2)/(8p^k)$ has no integral solutions.
\end{theorem}

\begin{remark}
By Theorem \ref{Thm3.2}, it is easy to see that the quadratic polynomials $2x^2+2y^2+P_7(z)$,
$4x^2+2y^2+P_5(z)$ are almost universal. Indeed, via computations, Sun \cite{S17} conjectured that the above
polynomials can represent all natural numbers over $\Z$.
\end{remark}

\begin{theorem}\label{Thm3.3}
Let $a,b,c\in\Z^+$ with $gcd(a,b,c)=1$, $\nu_p(a)\equiv\nu_p(b)\pmod 2$ and $\nu_2(a)\ge\nu_2(b)=0$.
Suppose that both $(1)$ and $(2)$ in Lemma \ref{Lem2.1} hold. Then $f_{a,b,c,p^k}$ is not almost universal
if and only if all of the following are satisfied:

{\rm (1)} Each prime divisor of $\mathcal{SF}(a'b'c')$ is congruent to $1$ modulo $4$ if
$\nu_2(a)\equiv\nu_2(b)\pmod2$, and is congruent to $1,3$ modulo $8$ if $\nu_2(a)\not\equiv\nu_2(b)\pmod2$.

{\rm (2)} $\nu_p(a)\equiv\nu_p(b)\equiv k\pmod 2$.

{\rm (3)} $4\nmid c$, $a'\equiv b'\pmod{2^{3-\nu_2(c)}}$, and
$$\begin{cases}p^kb'c'\equiv1\pmod4, \nu_2(a)\ge 2\ and\ \nu_2(a)\equiv 0\pmod 2&\mbox{if}\ 2\mid\mid c,\\p^kb'c'\equiv1\pmod8, \nu_2(a)\ge 3\ and\ \nu_2(a)\not\equiv0\pmod2&\mbox{if}\ 2\nmid c.\end{cases}$$

{\rm (4)} $2^{\nu_2(c)}\mathcal{SF}(a'b'c')c^{-1}$ is a quadratic residue modulo $p^k$, and
$ax^2+by^2+cP_{p^k+2}(z)=(2^{\nu_2(c)}\mathcal{SF}(a'b'c')-c(p^k-2)^2)/(8p^k)$ has no integral solutions.
\end{theorem}

\begin{remark}
According to Theorem \ref{Thm3.3}, one may easily verify that $4x^2+y^2+P_5(z)$ and $8x^2+y^2+P_5(z)$ are
almost universal.
\end{remark}

Now we turn to the cases when $\nu_p(a)\not\equiv\nu_p(b)\pmod2$. Let $a=p^{\nu_p(a)}a_0$ and $b=p^{\nu_p(b)}b_0$ with $p\nmid a_0b_0$. For each prime $q$, let $N_q(E)$ denote the group of local norm from $E_{\mathfrak{B}}$ to $\Q_q$, where $\mathfrak{B}$ is an extension of $q$ to $E$.

\begin{theorem}\label{Thm3.4}
Let $a,b,c\in\Z^+$ with $gcd(a,b,c)=1$, $\nu_p(a)\not\equiv\nu_p(b)\pmod 2$,
$\nu_2(a)\equiv\nu_2(b)\pmod2$ and $p\equiv3\pmod4$.
Suppose that both $(1)$ and $(2)$ in Lemma \ref{Lem2.1} hold. Then $f_{a,b,c,p^k}$ is not almost universal
if and only if all of the following are satisfied:

{\rm (1)}  For each prime divisor $q$ of $\mathcal{SF}(pa'b'c')$, we have $\left(\frac{-p}{q}\right)=1$.

{\rm (2)} $\left(\frac{2b_0c}{p}\right)=\left(\frac{2a_0c}{p}\right)=\left(\frac{a_0b_0}{p}\right)=1$.
\medskip

{\rm (3)} $pa'b'\equiv1\pmod{2^{3-\nu_2(c)}}$ and one of the following holds
\begin{align*}
{\rm (i)}\ &p\equiv7\pmod8,\\{\rm (ii)}\ &\nu_2(b)\not\equiv\nu_2(c)\pmod2\ and\ \nu_2(a)>\nu_2(b), \\{\rm (iii)}\ &\nu_2(b)\not\equiv\nu_2(c)\pmod2, \nu_2(a)=\nu_2(b)\ and\ a'b'\equiv3\pmod4.
\end{align*}

{\rm (4)} $2^{\nu_2(c)}\mathcal{SF}(pa'b'c')c^{-1}$ is a quadratic residue modulo $p^k$, and
$ax^2+by^2+cP_{p^k+2}(z)=(2^{\nu_2(c)}\mathcal{SF}(pa'b'c')-c(p^k-2)^2)/(8p^k)$ has no integral solutions.
\end{theorem}

\begin{remark}
By Theorem \ref{Thm3.4}, one may readily check that $x^2+15y^2+P_5(z)$, $x^2+11y^2+P_5(z)$, $x^2+7y^2+2P_5(z)$ and $4x^2+3y^2+P_5(z)$ are almost universal.
In fact, Sun \cite{S17} conjectured that the above polynomials can represent all natural numbers over $\Z$.
\end{remark}

\begin{theorem}\label{Thm3.5}
Let $a,b,c\in\Z^+$ with $gcd(a,b,c)=1$, $\nu_p(a)\not\equiv\nu_p(b)\pmod 2$,
$\nu_2(a)\equiv\nu_2(b)\pmod2$, $\nu_2(a)\ge\nu_2(b)$ and $p\equiv1\pmod4$.
Suppose that both $(1)$ and $(2)$ in Lemma \ref{Lem2.1} hold. Then $f_{a,b,c,p^k}$ is not almost universal
if and only if all of the following are satisfied:

{\rm (1)}  For each prime divisor $q$ of $\mathcal{SF}(pa'b'c')$, we have $\left(\frac{-p}{q}\right)=1$.

{\rm (2)} $\left(\frac{2b_0c}{p}\right)=\left(\frac{2a_0c}{p}\right)=\left(\frac{a_0b_0}{p}\right)=1$.
\medskip

{\rm (3)} $4\nmid c$, $pa'b'\equiv1\pmod{2^{3-\nu_2(c)}}$, and one of the following holds:
\begin{align*}
{\rm (i)}\ &2^{1+\nu_2(b)}b'c'\in N_2(\Q(\sqrt{-p}))\ and\ \nu_2(a)>\nu_2(b)\ge2,
\\{\rm (ii)}\ &b'c'\equiv1\pmod4,\ \nu_2(b)\in\{0,1\},\ \nu_2(c)\not\equiv\nu_2(b)\pmod2\ and\
\nu_2(a)>\nu_2(b),
\\{\rm (iii)}\ &\nu_2(a)=\nu_2(b)\ge1,\ and\ \begin{cases}b'c'\equiv1\pmod4\ &\mbox{if}\ p\equiv1\pmod8,\\b'c'\equiv 2+(-1)^{\nu_2(b)}\pmod4\ &\mbox{if}\ p\equiv5\pmod8.\end{cases}
\end{align*}

{\rm (4)} $2^{\nu_2(c)}\mathcal{SF}(pa'b'c')c^{-1}$ is a quadratic residue modulo $p^k$, and
$ax^2+by^2+cP_{p^k+2}(z)=(2^{\nu_2(c)}\mathcal{SF}(pa'b'c')-c(p^k-2)^2)/(8p^k)$ has no integral solutions.
\end{theorem}

\begin{remark}
In view of the above theorem, we can easily verify that $x^2+5y^2+P_5(z)$ is almost universal.
\end{remark}

The following theorem will cover all the remaining cases, for convenience, we set
$$\ve=\begin{cases}1&\mbox{if}\ \nu_p(b)\not\equiv k\pmod2,\\2&\mbox{otherwise}.\end{cases}$$

\begin{theorem}\label{Thm3.6}
Let $a,b,c\in\Z^+$ with $gcd(a,b,c)=1$, $\nu_p(a)\not\equiv\nu_p(b)\pmod 2$,
$\nu_2(a)\not\equiv\nu_2(b)\pmod2$, and $\nu_2(a)\ge\nu_2(b)$.
Suppose that both $(1)$ and $(2)$ in Lemma \ref{Lem2.1} hold. Then $f_{a,b,c,p^k}$ is not almost universal
if and only if all of the following are satisfied:

{\rm (1)}  For each prime divisor $q$ of $\mathcal{SF}(pa'b'c')$, we have $\left(\frac{-2p}{q}\right)=1$.

{\rm (2)} $\left(\frac{2a_0b_0}{p}\right)=\left(\frac{\ve b_0c}{p}\right)=1.$
\medskip

{\rm (3)} $2\nmid c$, $pa'b'\equiv1\pmod8$, $\nu_2(b)\ne1$ and one of the following holds:
\begin{align*}
{\rm (i)}\ &2^{1+\nu_2(b)}p^kb'c'\in N_2(\Q(\sqrt{-2p}))\ and\ \nu_2(a)>\nu_2(b)\ge2,
\\{\rm (ii)}\ &p^kb'c'\equiv p\pmod 8,\nu_2(b)=0\ and\ \nu_2(a)\ge3.
\end{align*}

{\rm (4)} $\mathcal{SF}(pa'b'c')c^{-1}$ is a quadratic residue modulo $p^k$, and
$ax^2+by^2+cP_{p^k+2}(z)=(\mathcal{SF}(pa'b'c')-c(p^k-2)^2)/(8p^k)$ has no integral solutions.
\end{theorem}

\begin{remark}
By Theorem \ref{Thm3.6}, one may easily check that $10x^2+y^2+P_5(z)$ is almost universal. Indeed, via Sun's
computation, it seems that this polynomial can represent each natural number over $\Z$.
\end{remark}

 Finally, we give an outline of this paper. In Section 2, we will give a brief overview of the theory of ternary quadratic forms which we need in our proofs, and we will prove the main results in Section 3.
\maketitle
\section{Some preparations}
\setcounter{lemma}{0}
\setcounter{theorem}{0}
\setcounter{corollary}{0}
\setcounter{remark}{0}
\setcounter{equation}{0}
\setcounter{conjecture}{0}

Let $L$ be a $\Z$-lattice on a positive definite ternary quadratic space $(V,B,Q)$ over $\Q$.
The discriminant of $L$ is denoted by $dL$.
Set $A$ be a
symmetric matrix, we write $L\cong A$ if $A$ is the gram matrix for $L$ with respect to some basis of $V$.
An $n\times n$ diagonal matrix with $a_1,...,a_n$ as the diagonal entries is written as $\langle a_1,...,a_n\rangle$ (any unexplained notation can be found in \cite{C,Ki,Oto}).

Given relatively prime positive integers $a,b,c$ and an odd prime $p$ not dividing $c$, let $f_{a,b,c,p^k}(x,y,z)=ax^2+by^2+cP_{p^k+2}(z)$. One may easily verify that an integer $n$ can
be represented by $f_{a,b,c,p^k}$ if and only if $8p^kn+c(p^k-2)^2$ can be represented by the coset
$M+v$, where $M$ is the $\Z$-lattice $\langle 8p^ka,8p^kb,4p^{2k}c\rangle$ in the orthogonal basis
$\{e_1,e_2,e_3\}$ and $v=-\frac{p^k-2}{2p^k}e_3$.

In order for $f_{a,b,c,p^k}$ to be almost universal, a necessary condition is that every integer of
the form $8p^kn+c(p^k-2)^2$ is represented by $\gen (M+v)$ (for the precise definition of $\gen(M+v)$,
the readers may consult \cite{B,X}).
Moreover, we have the following lemma.
\begin{lemma} \label{Lem2.1}
Every integer of the form $8p^kn+c(p^k-2)^2$ is represented by $\gen (M+v)$ if and only if we have the following
{\rm (1)} and {\rm (2)}.\\
{\rm (1)} For each prime $q\not\in\{2,p\}$, $M_q\cong\langle1,-1,-dM\rangle$.\\
{\rm (2)} Either $4\nmid c$ or both $4\mid\mid c$ and $2\mid\mid ab$.
\end{lemma}
\begin{proof}
As $P_{p^k+2}(z)=\frac{p^kz^2-(p^k-2)z}{2}$, by Hensel's Lemma, one may easily show that
$P_{p^k+2}(z)$ represents all $p$-adic integers over $\Z_p$, and note that $p\nmid c$, hence
each $p$-adic integer can be represented by $M_p+v$.

Now we consider prime $2$, using Hensel's Lemma, one may easily verify that $P_{p^k+2}(2z)$ can
represent all $2$-adic integers over $\Z_2$. If $2\nmid c$, clearly $cP_{p^k+2}(z)$ represents all
$2$-adic integers over $\Z_2$. If $2\mid\mid c$, then $\{cP_{p^k+2}(z): z\in\Z_2\}=2\Z_2$, since either
$a$ or $b$ is odd, we therefore have $f_{a,b,c,p^k}$ represents all $2$-adic integers over $\Z_2$. If $4\mid\mid c$, then $\{cP_{p^k+2}(z): z\in\Z_2\}=4\Z_2$.
Suppose $2\nmid ab$ or $4\mid ab$, then we have
$\#\{ax^2+by^2+4\Z: x,y\in\Z\}\le 3$ (where $\#S$ denotes the cardinality of a set $S$).
Suppose $2\mid\mid ab$, with the help of Hensel's Lemma, one may
easily check that $f_{a,b,c,p^k}$ represents all $2$-adic integers over $\Z_2$. If $8\mid c$, clearly,
$\#\{ax^2+by^2+8\Z: x,y\in\Z\}<7$, so the local conditions are not
satisfied.

Finally, if $q\not\in\{2,p\}$, then we have $M_q+v=M_q$.
If $M_q\cong\langle1,-1,-dM\rangle$, then
$M_q$ can represent all $q$-adic integers over $\Z_q$. Conversely, if each integer of the form $8p^kn+c(p^k-2)^2$ is represented by $\gen (M+v)$, clearly, $M_q$ represents all $q$-adic integers over
$\Z_q$. In particular $M_q$ must be isotropic. Hence, we have $M_q\cong\langle1,-1,-dM\rangle$.
\end{proof}

In view of Lemma \ref{Lem2.1}, we can simplify our problems. In fact, recall that
$M$ is the $\Z$-lattice $\langle 8p^ka,8p^kb,4p^{2k}c\rangle$ in the orthogonal basis
$\{e_1,e_2,e_3\}$, let $L$ be the $\Z$-lattice $\langle 8p^ka,8p^kb,c\rangle$ in the orthogonal
basis $\{e_1,e_2,\frac{1}{2p^k}e_3\}$. One may readily check the following result.

Suppose that both {\rm (1)} and {\rm (2)} in Lemma \ref{Lem2.1} hold. Then
$8p^kn+c(p^k-2)^2$ can be represented by the coset $M+v$ if and only if $8p^kn+c(p^k-2)^2$ can be
represented by $L$.

Now we need to introduce the theory of spinor exceptions. The readers can find relevant material in
\cite{SP00}. Let $W$ be a $\Z$-lattice on a quadratic space $V$ over $\Q$.
Suppose that $a\in \Z^+$ is represented by $\gen(W)$.
We call $a$ is a spinor exception of $\gen(W)$ if $a$ is represented by exactly half of
the spinor genera in the genus. R. Schulze-Pillot \cite{SP80} determined completely when $a$ is a spinor exception of $\gen(W)$. A. G. Earnest and J. S. Hisa and D. C. Hung \cite{EHH} found a similar characterization of primitive spinor exceptional numbers. We also need the following lemma.

\begin{lemma}\label{Lem2.2}
{\rm (i)} For each integer $m,n\ne0$, there are infinitely many rational primes
that split in $\Q(\sqrt{m})$ and are congruent to $1$ modulo $n$.

{\rm (ii)} Given an odd prime $p$ and a positive integer $k$, let $E\in\{\Q(\sqrt{-m}): m=1,2,p,2p\}$,
then there are infinitely many rational primes that are inert in $E$ and congruent to $\pm1$ modulo $p^k$.
\end{lemma}
\begin{proof}
(i) Set $K=\Q(\sqrt{m})$ with absolute discriminant $d_K$, by Kronecker-Weber's Theorem, we have
$\Q\subseteq\Q(\sqrt{m})\subseteq\Q(\zeta_{d_K})\subseteq\Q(\zeta_{nd_K})$ (where $\zeta_l=e^{2\pi i/l})$.
By Dirichlet's Theorem, there are infinitely many rational primes that are congruent to $1$ modulo $n|d_K|$, since these primes totally split in $\Q(\zeta_{nd_K})$, they also split in $Q(\sqrt{m})$.
(ii) With the help of the Chinese Remainder Theorem, one may easily get the desired results
(for more details, the readers may consult the excellent book \cite{Lang}).
\end{proof}

Recall that $L$ is the $\Z$-lattice $\langle 8p^ka,8p^kb,c\rangle$, let $N$ be its
level. Suppose that a square-free positive integer $t$ is a primitive spinor exception
of $\gen(L)$ such that $t\equiv c\pmod8$ and $tc^{-1}$ is a quadratic residue modulo $p^k$. If $t$ is not represented by the spinor genus of
$L$, R. Schulze-Pillot \cite{SP00} proved that for each positive integer $m$ with $gcd(m,N)=1$, $tm^2$
is not represented by the spinor genus of $L$ provided that each prime divisor of $m$ splits in $E=\Q(\sqrt{-tdL})$. If $t$ is represented by the spinor genus of $L$, but not represented by $L$, then for each positive integer $m$ with $gcd(m,N)=1$, $tm^2$ is not primitively represented by the spinor genus of $L$ provided that at least one prime divisor of $m$ is inert in $E$. In particular, if each prime factor of $m$ is inert in $E$, then one may easily verify that $tm^2$ is not represented by $L$.
We will show below that $E$ must be in the set $\{\Q(\sqrt{-D}): D=1,2,p,2p\}$. Therefore, by the Chinese Remainder Theorem, there exists a prime $q_0$ not dividing $N$ such that $tq_0^2\equiv c(p^k-2)^2\pmod{8p^k}$.
In view of the above, by Lemma \ref{Lem2.2}, we can find infinitely many primes $q$ with $q^2\equiv1\pmod{8p^k}$ such that $tq_0^2q^2$ is not represented by $L$. Hence, $f_{a,b,c,p^k}$ is not almost universal.

On the other hand, suppose that both $(1)$ and $(2)$ in Lemma \ref{Lem2.1} hold,
it is easy to see that each positive integer of the form $8p^kn+c(p^k-2)^2$ can be represented by $\gen(L)$
primitively. Assume that each primitive spinor exception of $\gen(L)$ in the arithmetic progression $\{c(p^k-2)^2+8p^kn: n\in\N\}$ can be represented by $L$, then by \cite[the Corollary of Theorem 3]{DR}, one can easily verify that $8p^kn+c(p^k-2)^2$ can be represented by $L$ provided that
$n$ is sufficiently large (for more details, the readers may see \cite{DR,KS,WS}).

Now we consider the quadratic field $E=\Q(\sqrt{-tdL})$, let $\theta$ denote the spinor norm map
and $N_p(E)$ denote the group of local norm from $E_{\mathfrak{B}}$ to $\Q_p$, where $\mathfrak{B}$ is
an extension of $p$ to $E$.
By virtue of the proof of Lemma \ref{Lem2.1},
for all prime $q\not\in\{2,p\}$, $L_q$ represents all $q$-adic integers over $\Z_q$, hence we have
$\Z_q^{\times}\subseteq\theta(O^+(L_q))\subseteq N_q(E)$. Therefore, each prime $q\not\in\{2,p\}$ is
unramified in $E$, and hence $E\in \{\Q(\sqrt{-D}): D=1,2,p,2p\}$.

The theory of spinor exceptions involves the computation of the spinor norm groups of local integral rotations and the relative spinor norm groups.
The reader can find relevant formulae in \cite{EH75,EH78,EHH,Ke56}. A correction of some of these formulae can be found in \cite{WKCOH}.
\maketitle
\section{Proofs of Main Results}
\setcounter{lemma}{0}
\setcounter{theorem}{0}
\setcounter{corollary}{0}
\setcounter{remark}{0}
\setcounter{equation}{0}
\setcounter{conjecture}{0}

In this section, we shall prove our main results. Recall that $L$ is the $\Z$-lattice $\langle 8p^ka,8p^kb,c\rangle$ with $\nu_2(a)\ge\nu_2(b)$. The squarefree part of an integer $m$ is denoted by $\mathcal {SF}(m)$ and the odd part of $m$ is denoted by $m'$.
\medskip

\noindent{\it \bf Proof of Theorem 1.1}

We shall show that if conditions $(1),(2),(3)$ are all satisfied, then
$t=\mathcal{SF}(a'b'c')$
is a primitive spinor exception of $\gen(L)$ (it is easy to see that $t$ can be primitively represented by $L$ locally). Set $E=\Q(\sqrt{-tdL})$, one may easily verify that
$E\in\{\Q(\sqrt{-m}): m=1,2\}$.

When $q\mid \mathcal{SF}(a'b'c')$, by Lemma \ref{Lem2.1}, we have $L_q\cong\langle1,-1,-dL\rangle$.
Thus, by \cite[Theorem 1(a)]{EHH}, we have
$\theta(O^+(L_q))\subseteq N_q(E)=\theta^*(L_q,t)$ (where $\theta^*(L_q,t)$ is the primitively relative spinor norm group) if and only if $-tdL\in\Q_q^{\times2}$, i.e., $(1)$ is satisfied.

When $q\not\in\{2,p\}$ and $q\nmid t$, clearly, $q$ is not ramified in $E$. Moreover, by Lemma \ref{Lem2.1},
we have $L_q\cong M_q\cong\langle 1,-1,-dL\rangle.$
Therefore, we have $\theta(O^+(L_q))\subseteq N_q(E)$ and $\theta^*(L_q,t)=N_q(E)$
(by \cite[Theorem 1(a)]{EHH}).

When $q=p$, we have $L_p\cong\langle v_1,v_2p^{e},v_3p^{f}\rangle$ where $v_i\in\Z_p^{\times}$ and
$e=k+\nu_p(b)$, $f=k+\nu_p(a)$. Since $p$ is not ramified in $E$, by \cite[Theorem 1(a)]{EHH},
we have $\theta(O^+(L_p))\subseteq N_p(E)$ and $\theta^*(L_p,t)=N_p(E)$ if and only if $(2)$ is satisfied.

When $q=2$, then $L_2\cong\langle c', 2^rp^kb',2^sp^ka'\rangle$, where $r=3+\nu_2(b)$ and $s=3+\nu_2(a)$.
Now, suppose that $(3)$ is satisfied.

We first consider the case when $r<s$.
Let $U=\langle c',2^rp^kb'\rangle$ and $W=\langle 2^rp^kb',2^sp^ka'\rangle$.
By \cite[Theorem 2.7]{EH75}, then we have
$$\theta(O^+(L_2))=Q(\mathcal{P}(U))Q(\mathcal{P}(W))\Q_2^{\times2},$$ where
$\mathcal{P}(U)$ $(\mathcal{P}(W))$ is the set of primitive anisotropic vectors in $U$ $(V)$ whose associate symmetries are in $O(U)$ $(O(W))$. Note that
$$Q(\mathcal{P}(U))\Q_2^{\times2}=c'\theta(O^+(\langle1,2^rp^kb'c'\rangle)),$$ and
$$Q(\mathcal{P}(W))\Q_2^{\times2}=2^rp^kb'\theta(O^+(\langle1,2^{s-r}a'b'\rangle)).$$
The formulae for the spinor norm group of non-modular binary $\Z_2$-lattice are available in \cite[1.9]{EH75}.
Hence, one can easily verify the following results:
$$\theta(O^+(L_2))=\begin{cases}2^rp^kb'c'\{1,2,3,6\}\Q_2^{\times2}\cup\{1,2,3,6\}\Q_2^{\times2}&\mbox{if}\ s-r\in\{1,3\},\\2^rp^kb'c'\{1,5\}\Q_2^{\times2}\cup\{1,5\}\Q_2^{\times2}&\mbox{if}\ s-r\in\{2,4\} ,\\2^rp^kb'c'\{1,2^{s-r}\}\Q_2^{\times2}\cup\{1,2^{s-r}\}\Q_2^{\times2}&\mbox{if}\ s-r\ge5.\end{cases}$$
According to the above results, by \cite[Theorem 2(b)]{EHH}, we have
$\theta(O^+(L_2))\subseteq N_2(E)$ and $\theta^*(L_2,t)=N_2(E)$.

Now we consider the case when $r=s$, by \cite[1.2]{EH78}, we have
$$\theta(O^+(L_2))=\begin{cases}\Q_2^{\times}&\mbox{if}\ \text{(3) is not satisfied},\\\{\gamma\in\Q_2^{\times}: (\gamma,-1)_2=1\}= N_2(E)&\mbox{otherwise},\end{cases}$$
where $(\ ,\ )_2$ is the Hilbert Symbol in $\Q_2$. Hence, by \cite[Theorem 2(b)]{EHH}, we have
$\theta(O^+(L_2))\subseteq N_2(E)$ and $\theta^*(L_2,t)=N_2(E)$.

In view of the above, $t=\mathcal{SF}(a'b'c')$
is a primitive spinor exception of $\gen(L)$ and since $a'\equiv b'\pmod 8$, we also have $t\equiv c\pmod 8$.
If $(4)$ is satisfied, by the discussions following the proof of Lemma \ref{Lem2.2}, it is easy to see that $f_{a,b,c,p^k}$ is not almost universal.

Conversely, according to the results in \cite[Theorem 1(a) and Theorem 2(b)]{EHH}, if one of $(1),(2),(3)$ is not
satisfied, then $\gen(L)$ does not have any spinor exceptions in the arithmetic progression $\{c(p^k-2)^2+8p^kn: n\in\N\}$. Hence by the discussions following the proof
of Lemma \ref{Lem2.2}, we have $f_{a,b,c,p^k}$ is almost universal. Assume now that the conditions
$(1),(2),(3)$ are all satisfied. It might worth mentioning here that it is necessary in $(3)$ to require that
$a'\equiv b'\pmod 8$, since by the discussions following the proof of Lemma \ref{Lem2.2}, if $f_{a,b,c,p^k}$
is not almost universal, then there exists some odd integer $l$ such that
$tl^2=\mathcal{SF}(a'b'c')l^2\equiv c(p^k-2)^2\pmod {8p^k}$,
by Lemma \ref{Lem2.1}(i) and the fact that $p\nmid c$,
it is easy to see
that $gcd(c',a'b')=1$. Hence, we must have $a'\equiv b'\pmod 8$.
If $tc^{-1}$ is not a quadratic residue modulo $p^k$, then there does not exist any spinor exceptions of $\gen(L)$ in the arithmetic progression $\{c(p^k-2)^2+8p^kn: n\in\N\}$.
If $tc^{-1}$ is a quadratic residue modulo $p^k$ and $ax^2+by^2+cP_{p^k+2}(z)=\mathcal{SF}(a'b'c')$ has an integral solution, then each spinor exception of $\gen(L)$ in the arithmetic progression $\{c(p^k-2)^2+8p^kn: n\in\N\}$
can be represented by $L$. By the discussions following the proof of Lemma \ref{Lem2.2}, we have
$f_{a,b,c,p^k}$ is almost universal.

In view of the above, we complete the proof of Theorem \ref{Thm3.1}.\qed
\medskip

\noindent{\it \bf Proof of Theorem 1.2}

As in the proof of Theorem \ref{Thm3.1}, we have
$\theta(O^+(L_q))\subseteq N_q(E)$ and $\theta^*(L_q,t)=N_q(E)$ for each prime $q\ne2$
if and only if both $(1)$ and $(2)$ hold.

When $q=2$, then $L_2\cong\langle c', 2^rp^kb',2^sp^ka'\rangle$, where $r=3+\nu_2(b)=4$ and $s=3+\nu_2(a)$.
Now, suppose that $(3)$ is satisfied.

We first consider the case when $r<s$, if $s-r\in\{1,3\}$, then $L_2$ is of {\it Type} E and hence
$\theta(O^+(L_2))=\Q_2^{\times}$ (by \cite[1.1]{EH78}). If $s-r\not\in\{1,3\}$,
let $U=\langle c',2^4p^kb'\rangle$ and $W=\langle 2^4p^kb',2^sp^ka'\rangle$, then $$\theta(O^+(L_2))=Q(\mathcal{P}(U))Q(\mathcal{P}(W))\Q_2^{\times2}.$$ One
may easily obtain the following results:
$$\theta(O^+(L_2))=\begin{cases}\Q_2^{\times}&\mbox{if}\ s-r\in\{1,3\},\\p^kb'c'\{1,5\}\Q_2^{\times2}\cup\{1,5\}\Q_2^{\times2}&\mbox{if}\ s-r\in\{2,4\} ,\\p^kb'c'\{1,2^s,5,5\times2^s\}\Q_2^{\times2}\cup\{1,2^s,5,5\times2^s\}\Q_2^{\times2}&\mbox{if}\ s-r\ge5.\end{cases}$$
Hence, we have $\theta(O^+(L_2))\subseteq N_2(E)$ and $\theta^*(L_2,t)=N_2(E)$.

Now we consider the case when $r=s=4$, we have
$$\theta(O^+(L_2))=\begin{cases}\Q_2^{\times}&\mbox{if}\ \text{(3) is not satisfied},\\\{\gamma\in\Q_2^{\times}: (\gamma,-1)_2=1\}= N_2(E)&\mbox{otherwise}.\end{cases}$$
Hence, we have $\theta(O^+(L_2))\subseteq N_2(E)$ and $\theta^*(L_2,t)=N_2(E)$.

In view of the above, we have $t=\mathcal{SF}(a'b'c')$
is a primitive spinor exception of $\gen(L)$ and since $a'\equiv b'\pmod 8$, we also have $t\equiv c\pmod 8$.
As in the proof of Theorem \ref{Thm3.1}, if $(4)$ is satisfied, we have $f_{a,b,c,p^k}$ is not almost
universal.

The proof of the converse is similar to the proof in Theorem \ref{Thm3.1}.

This completes our proof of Theorem \ref{Thm3.2}.\qed
\medskip

\noindent{\it \bf Proof of Theorem 1.3}

In light of the proof of Theorem \ref{Thm3.1}, we have
$\theta(O^+(L_q))\subseteq N_q(E)$ and $\theta^*(L_q,t)=N_q(E)$ for each prime $q\ne2$
if and only if both $(1)$ and $(2)$ hold.

Now we consider the prime $q=2$, if $(3)$ is satisfied, let $r=3-\nu_2(c)$ and $s=3-\nu_2(c)+\nu_2(a)$.

If $\nu_2(c)=2$, then by Lemma \ref{Lem2.1}, we have $L_2^{1/4}\cong\langle c', 2p^kb',2^2p^ka'\rangle$. Then $L_2$ is of {\it Type} E and hence
$\theta(O^+(L_2))=\Q_2^{\times}\not\subseteq N_2(E)$. So, we must have $4\nmid c$. Hence, we
divide our proof into two cases.

{\it Case} 1. $\nu_2(c)=0$.

$L_2\cong\langle c', 2^3p^kb',2^sp^ka'\rangle$ in this case.
If $s=3$, then by \cite [1.2]{EH78}, we have $\theta(O^+(L_2))=\Q_2^{\times}\not\subseteq N_2(E)$,
if $s\in\{4,5,7\}$, $L_2$ is of {\it Type} E and hence $\theta(O^+(L_2))=\Q_2^{\times}\not\subseteq N_2(E)$.
Therefore, we just need consider the case when $3<s$ and $s\not\in\{4,5,7\}$. Let $U=\langle c',2^3p^kb'\rangle$ and $W=\langle 2^3p^kb',2^sp^ka'\rangle$, then $$\theta(O^+(L_2))=Q(\mathcal{P}(U))Q(\mathcal{P}(W))\Q_2^{\times2}.$$
Since
$$Q(\mathcal{P}(U))\Q_2^{\times2}=c'\theta(O^+(\langle1,2^3p^kb'c'\rangle)),$$
and
$$Q(\mathcal{P}(W))\Q_2^{\times2}=2^3p^kb'\theta(O^+(\langle1,2^{s-3}a'b'\rangle)).$$
If $a'\equiv b'\pmod 8$, by \cite[1.9]{EH75}, we have
$$Q(\mathcal{P}(U))\Q_2^{\times2}=c'\{\gamma\in\Q_2^{\times}: (\gamma,-2p^kb'c')_2=1\},$$
and
$$Q(\mathcal{P}(W))=\begin{cases}2p^kb'\{1,2,3,6\}\Q_2^{\times2}&\mbox{if}\ s=6,\\2p^kb'\Q_2^{\times2}\cup2^sp^kb'\Q_2^{\times2}&\mbox{if}\ s\ge 8.\end{cases}$$

Assume first that $\nu_2(a)\equiv 0\pmod 2$, if $p^kb'c'\not\equiv 1\pmod 4$, then we have $2p^kb'c'\Q_2^{\times2}
\in\theta(O^+(L_2))$. Hence $\theta(O^+(L_2))\not\subseteq N_2(E)=\{1,2,5,10\}\Q_2^{\times2}$. If $p^kb'c'\equiv 1\pmod 4$, we have $3c'\in Q(\mathcal{P}(U))$ and hence $6p^kb'c'\in\theta(O^+(L_2))\not\subseteq N_2(E)=\{1,2,5,10\}\Q_2^{\times2}$. Therefore, in Case 1,
we always have
$\theta(O^+(L_2))\not\subseteq N_2(E)$ if $\nu_2(a)\equiv 0\pmod 2$.

Assume now that $\nu_2(a)\not\equiv 0\pmod 2$, by the above formulae, one may easily verify that $\theta(O^+(L_2))\subseteq N_2(E)=\{1,2,3,6\}\Q_2^{\times2}$ if and
only if $p^kb'c'\equiv 1\pmod 8$ and $\nu_2(a)\ge 3$.

{\it Case} 2. $\nu_2(c)=1$.

We have $L_2^{1/2}\cong\langle c', 2^2p^kb',2^sp^ka'\rangle$ and $a'\equiv b'\pmod 4$ in this case. If $s=2$ or $s\in\{3,5\}$, by \cite[1.1 and 1.2]{EH78}, we have $\theta(O^+(L_2))\not\subseteq N_2(E)$. Hence, we just need consider the case when $s\not\in\{2,3,5\}$.
Let $U=\langle c',2^2p^kb'\rangle$ and $W=\langle 2^2p^kb',2^sp^ka'\rangle$, we have
$$\theta(O^+(L_2))=Q(\mathcal{P}(U))Q(\mathcal{P}(W))\Q_2^{\times2},$$
and
$$Q(\mathcal{P}(U))\Q_2^{\times2}=c'\{\gamma\in\Z_2^{\times}\Q_2^{\times2}: (\gamma,-p^kb'c')_2=1\},$$
and
$$Q(\mathcal{P}(W))\Q_2^{\times2}=\begin{cases}\{p^kb', 5p^kb'\}\Q_2^{\times2}&\mbox{if}\ s\in\{4,6\},\\p^kb'\Q_2^{\times2}\cup2^sp^ka'\Q_2^{\times2}&\mbox{if}\ s\ge7.\end{cases}$$
According to the above results, one may easily verify that $\theta(O^+(L_2))\subseteq N_2(E)$ if and only if
$\nu_2(a)\equiv 0\pmod 2$, $\nu_2(a)\ge 2$ and $p^kb'c'\equiv 1\pmod 4$.

Finally, as in the proof of Theorem \ref{Thm3.1}, if $(4)$ is satisfied, we have $f_{a,b,c,p^k}$ is not almost
universal.

Conversely, according to the results in \cite[Theorem 1(a) and Theorem 2(b)]{EHH}, if one of $(1),(2),(3)$ is not satisfied, then $\gen(L)$ does not have any spinor exceptions in the arithmetic progression $\{c(p^k-2)^2+8p^kn: n\in\N\}$. Hence by the discussions following the proof
of Lemma \ref{Lem2.2}, we have $f_{a,b,c,p^k}$ is almost universal. Assume now that the conditions
$(1),(2),(3)$ are all satisfied. It might worth mentioning here that it is necessary in $(3)$ to require that
$a'\equiv b'\pmod {2^{3-\nu_2(c)}}$, since by the discussions following the proof of Lemma \ref{Lem2.2}, if $f_{a,b,c,p^k}$
is not almost universal, then there exists some odd integer $l$ such that
$tl^2=2^{\nu_2(c)}\mathcal{SF}(a'b'c')l^2\equiv c(p^k-2)^2\pmod {8p^k}$,
by Lemma \ref{Lem2.1}(i) and the fact that $p\nmid c$, it is easy to see
that $gcd(c',a'b')=1$. Hence, we must have $a'\equiv b'\pmod {2^{3-\nu_2(c)}}$.
The remaining proof of the converse is similar to the proof in Theorem \ref{Thm3.1}.

This completes the proof of Theorem \ref{Thm3.3}.\qed
\medskip

\noindent{\it \bf Proof of Theorem 1.4}

We shall show that if conditions $(1),(2),(3)$ are all satisfied, then $w=2^{\nu_2(c)}\mathcal{SF}(pa'b'c')$
is a primitive spinor exception of $\gen(L)$. Set $E=\Q(\sqrt{-wdL})$, one may easily verify that
$E=\Q(\sqrt{-p})$.

When $q\mid\mathcal{SF}(pa'b'c')$, clearly, $q\not\in\{2,p\}$ and $q$ is not ramified in $E$, by Lemma \ref{Lem2.1}, we have
$L_q\cong M_q\cong\langle 1,-1,-dL\rangle$. Hence, by \cite[Theorem 1(a)]{EHH},
we have
$\theta(O^+(L_q))\subseteq N_q(E)=\theta^*(L_q,w)$ if and only if $-wdL\in\Q_q^{\times2}$,
i.e., $(1)$ is satisfied.

When $q\not\in\{2,p\}$ and $q\nmid w$, clearly, $q$ is not ramified in $E$. Moreover, by Lemma \ref{Lem2.1},
we have $L_q\cong M_q\cong\langle 1,-1,-dL\rangle$. Therefore, we have $\theta(O^+(L_q))\subseteq N_q(E)$ and $\theta^*(L_q,t)=N_q(E)$
(by \cite[Theorem 1(a)]{EHH}).

When $q=p$, by \cite[Satz 3]{Ke}, we have $$\theta(O^+(L_p))=
\\\Q_p^{\times2}\cup p^{\nu_p(a)+k}2a_0c\Q_p^{\times2}\cup p^{\nu_p(b)+k}2b_0c\Q_p^{\times2}\cup pa_0b_0\Q_p^{\times2}.$$
By \cite[Theorem 1(b)]{EHH}, we have
$$\theta(O^+(L_p))\subseteq N_p(E)=\theta^*(L_p,w)=\{1,p\}\Q_p^{\times2}$$
if and only if $(2)$ is satisfied.

When $q=2$, then $L_2^{1/2^{\nu_2(c)}}\cong\langle c',2^rp^kb',2^sp^ka'\rangle$, where
$r=3+\nu_2(b)-\nu_2(c)$ and $s=3+\nu_2(a)-\nu_2(c)$. If $p\equiv7\pmod8$, then $-wdL\in\Q_2^{\times2}$. Therefore, $\theta(O^+(L_2))\subseteq N_2(E)=\theta^*(L_2,w)=\Q_2^{\times}$.

If $p\equiv3\pmod8$, note that $2$ is unramified in $E$, by \cite[Theorem 2(a)]{EHH}, we have $\theta(O^+(L_2))\subseteq
N_2(E)$ and $\theta^*(L_2,w)=N_2(E)$ if and only if the Jordan components of $L_2$ all have even orders
(a $\Z_2$-lattice $M$ has even order
if $\nu_2(Q(v))$ is even for each primitive vector $v\in M$ which gives rise to an integral symmetry of $M$).
If $r<s$, then we must have $r\equiv s\equiv 0\pmod2$. If $r=s$, by \cite[Propositon 3.2(1)]{EH75}, one may
easily verify that $2^r\langle p^kb',p^ka'\rangle$ has even order if and only if
$r\equiv0\pmod2$ and $a'b'\equiv3\pmod4$.

In view of the above, it is easy to see that $w=2^{\nu_2(c)}\mathcal{SF}(pa'b'c')$ is a primitive spinor
exception of $\gen(L)$, and since $pa'b'\equiv 1\pmod {2^{3-\nu_2(c)}}$, we also have $w\equiv c\pmod 8$.
If $(4)$ is satisfied, by the discussions following the proof of Lemma \ref{Lem2.2}, we have $f_{a,b,c,p^k}$ is not almost universal.

Conversely, as in the proof in Theorem \ref{Thm3.1}, if one of the $(1),(2),(3)$ is not satisfied, then
$\gen(L)$ does not have any spinor exceptions in the arithmetic progression $\{c(p^k-2)^2+8p^kn: n\in\N\}$.
Assume now that the conditions
$(1),(2),(3)$ are all satisfied. It might worth mentioning here that it is necessary in $(3)$ to require that
$pa'b'\equiv 1\pmod {2^{3-\nu_2(c)}}$, since by the discussions following the proof of Lemma \ref{Lem2.2}, if $f_{a,b,c,p^k}$
is not almost universal, then there exists some odd integer $l$ such that
$wl^2=2^{\nu_2(c)}\mathcal{SF}(pa'b'c')l^2\equiv c(p^k-2)^2\pmod {8p^k}$,
by Lemma \ref{Lem2.1}(i) and the fact that $p\nmid c$, it is easy to see
that $gcd(c',pa'b')=1$. Hence, we must have $pa'b'\equiv 1\pmod {2^{3-\nu_2(c)}}$.
If $wc^{-1}$ is not a quadratic residue modulo $p^k$, then there does not exist any spinor exceptions of $\gen(L)$ in the arithmetic progression $\{c(p^k-2)^2+8p^kn: n\in\Z^+\}$.
If $wc^{-1}$ is a quadratic residue modulo $p^k$ and $ax^2+by^2+cP_{p^k+2}(z)=2^{\nu_2(c)}\mathcal{SF}(pa'b'c')$ has an integral solution, then each spinor exception of $\gen(L)$ in the arithmetic progression $\{c(p^k-2)^2+8p^kn: n\in\N\}$
can be represented by $L$. By the discussions following the proof of Lemma \ref{Lem2.2}, we have
$f_{a,b,c,p^k}$ is almost universal.

This completes the proof of Theorem \ref{Thm3.4}.\qed
\medskip

\noindent{\it \bf Proof of Theorem 1.5}

Set $w=2^{\nu_2(c)}\mathcal{SF}(pa'b'c')$ and $E=\Q(\sqrt{-wdL})=\Q(\sqrt{-p})$.
By the proof of Theorem \ref{Thm3.4}, then for each prime $q\ne2$, we have
$\theta(O^+(L_q))\subseteq N_q(E)$ and $\theta^*(L_q,t)=N_q(E)$ if and only if both $(1)$ and $(2)$ hold.

Now we consider the prime $q=2$. Note that $2$ is ramified in $E$ and $-p\not\in\Q_2^{\times2}$.
It is easy to see that
$$N_2(\Q(\sqrt{-p}))=\{1,\ 5,\ 1+p,\ 5\times(1+p)\}\Q_2^{\times2}.$$ If $(3)$ is satisfied,
set
$L_2=2^{\nu_2(c)}\langle c',2^rp^kb',2^sp^ka'\rangle$, $U=\langle c',2^rp^kb'\rangle$,
and $W=\langle 2^rp^kb',2^sp^ka'\rangle$,
where $r=3-\nu_2(c)+\nu_2(b)$ and $s=3-\nu_2(c)+\nu_2(a)$. According to
the different ways to compute $\theta(O^+(L_2))$, we will divide the remaining proof into the following four cases.

{\it Case} 1. $\nu_2(a)>\nu_2(b)\ge2$.

In this case, note that $a'b'\equiv p\pmod 8$, by \cite[Theorem 2.7]{EH75}, we have
$$\theta(O^+(L_2))=Q(\mathcal{P}(U))Q(\mathcal{P}(W))\Q_2^{\times2},$$
and
$$Q(\mathcal{P}(U))\Q_2^{\times2}=c'\theta(O^+(\langle1,2^rp^kb'c'\rangle)),$$ and
$$Q(\mathcal{P}(W))\Q_2^{\times2}=2^rp^kb'\theta(O^+(\langle1,2^{s-r}a'b'\rangle)).$$
Thus, one may easily obtain the following results:
$$Q(\mathcal{P}(U))\Q_2^{\times2}=\{c',2^rp^kb'\}\Q_2^{\times2},$$
and
$$Q(\mathcal{P}(W))\Q_2^{\times2}=\begin{cases}2^rp^kb'\{1,5\}\Q_2^{\times2}&\mbox{if}\ s-r\in\{2,4\},\\2^rp^kb'\{1,p\}\Q_2^{\times2}&\mbox{if}\ s-r\ge5.\end{cases}$$
Thus, we have
$$\theta(O^+(L_2)=\begin{cases}2^rp^kb'c'\{1,5\}\Q_2^{\times2}\cup\{1,5\}\Q_2^{\times2}&\mbox{if}\ s-r\in\{2,4\},\\2^rp^kb'c'\{1,p\}\Q_2^{\times2}\cup\{1,p\}\Q_2^{\times2}&\mbox{if}\ s-r\ge5.\end{cases}$$
Hence, by \cite[Theorem 2(b)]{EHH} we have
$\theta(O^+(L_2))\subseteq N_2(E)$ and $\theta^*(L_2,w)=N_2(E)$ if and only if
$2^{1+\nu_2(b)}p^kb'c'\in N_2(E)$, i.e., (i) of $(3)$ is satisfied.

{\it Case} 2. $\nu_2(a)>\nu_2(b)=1$.

In this case, we have $r=4$, note that $L_2$ is not of {\it Type} E and $a'b'\equiv p\pmod 8$. Hence, by \cite[Theorem 2.7]{EH75}, we have
$$\theta(O^+(L_2))=Q(\mathcal{P}(U))Q(\mathcal{P}(W))\Q_2^{\times2},$$
and
$$Q(\mathcal{P}(U))\Q_2^{\times2}=c'\theta(O^+(\langle1,2^4p^kb'c'\rangle)),$$
and
$$Q(\mathcal{P}(W))\Q_2^{\times2}=2^4p^kb'\theta(O^+(\langle1,2^{s-4}a'b'\rangle)).$$
Then, note that $p\equiv 1\pmod 4$, one may easily get the following results:
$$Q(\mathcal{P}(U))\Q_2^{\times2}=\{c',5c'\}\Q_2^{\times2}\cup\{b',5b'\}\Q_2^{\times2},$$
and
$$Q(\mathcal{P}(W))\Q_2^{\times2}=\begin{cases}b'\{1,5\}\Q_2^{\times2}&\mbox{if}\ s-r\in\{2,4\},\\p^kb'\{1,p\}\Q_2^{\times2}&\mbox{if}\ s-r\ge5.\end{cases}$$
Thus we have
$$\theta(O^+(L_2)=\begin{cases}b'c'\{1,5\}\Q_2^{\times2}\cup\{1,5\}\Q_2^{\times2}&\mbox{if}\ s-r\in\{2,4\},\\b'c'\{1,5\}\Q_2^{\times2}\cup\{1,5\}\Q_2^{\times}&\mbox{if}\ s-r\ge5.\end{cases}$$
Therefore, $\theta(O^+(L_2))\subseteq N_2(E)$ and $\theta^*(L_2,w)=N_2(E)$ if and only if $b'c'\equiv1\pmod4$.

{\it Case} 3. $\nu_2(a)>\nu_2(b)=0$.

In the present case, if $4\mid\mid c$, then
$L_2^{1/4}\cong\langle c', 2p^kb',2^2p^ka'\rangle$,
by \cite[1.1]{EH78}, $L_2$ is of {\it Type} E and hence
$\theta(O^+(L_2))=\Q_2^{\times}\not\subseteq N_2(E)$.

If $2\nmid c$, then $L_2\cong\langle c',2^3p^kb',2^sp^ka'\rangle$,
by \cite[Theorem 2(b)(iv)]{EHH}, we have $\theta^*(L_2,w)\not\subseteq N_2(E)$.

If $2\mid\mid c$, then $L_2^{1/2}\cong\langle c', 2^2p^kb',2^sp^ka'\rangle$,
$U=\langle c',2^2p^kb'\rangle$ and $W=\langle 2^2p^kb',2^sp^ka'\rangle$,
note that $L_2$ is not of {\it Type} E,
by \cite[Theorem 2.7]{EH75}, we have
$$\theta(O^+(L_2))=Q(\mathcal{P}(U))Q(\mathcal{P}(W))\Q_2^{\times2},$$
and
$$Q(\mathcal{P}(U))\Q_2^{\times2}=c'\{\gamma\in\Z_2^{\times}\Q_2^{\times2}: (\gamma,-p^kb'c')_2=1\},$$
and
$$Q(\mathcal{P}(W))\Q_2^{\times2}=\begin{cases}p^kb'\{1,5\}\Q_2^{\times2}&\mbox{if}\ s\in\{4,6\},\\p^kb'\Q_2^{\times2}\cup p^ka'\Q_2^{\times2}&\mbox{if}\ s\ge7.\end{cases}$$
By the above results and note that $a'\equiv b'\pmod 4$, then we have
$$\theta(O^+(L_2)=\begin{cases}\Z_2^{\times}\Q_2^{\times2}&\mbox{if}\ b'c'\equiv 3\pmod 4,\\\{1,5\}\Q_2^{\times}&\mbox{if}\ b'c'\equiv 1\pmod 4.\end{cases}$$
Hence, we have
$\theta(O^+(L_2))\subseteq N_2(E)$ and $\theta^*(L_2,w)=N_2(E)$ if and only if $b'c'\equiv1\pmod4$.

In view of the Case 3 and Case 4, when $\nu_2(b)\in\{0,1\}$ and $\nu_2(a)>\nu_2(b)$,
$\theta(O^+(L_2))\subseteq N_2(E)$ and $\theta^*(L_2,w)=N_2(E)$ if and only if (ii) of $(3)$ is satisfied.

{\it Case} 4. $\nu_2(a)=\nu_2(b)$.

In the present case, $L_2^{1/2^{\nu_2(c)}}\cong\langle c'\rangle\perp2^r\langle p^kb',p^ka'\rangle$,
if $r\le3$, by \cite[1.2]{EH78}, we have $\theta(O^+(L_2))\not\subseteq N_2(E)$, hence we must have
$\nu_2(a)=\nu_2(b)\ge1$.

If $p\equiv1\pmod8$, then we have
$$N_2(\Q(\sqrt{-p}))=\{1,2,5,10\}\Q_2^{\times2}.$$ Note that $a'b'\equiv 1\pmod 8$,
by \cite[1.2]{EH78}, one may easily obtain the following result:
$$\theta(O^+(L_2))=\begin{cases}\Q_2^{\times}\ \text{or}\ \Z_2^{\times}\Q_2^{\times2}&\mbox{if}\ b'c'\not\equiv 1\pmod 4 ,\\\{\gamma\in\Q_2^{\times}: (\gamma,-1)_2=1\}= N_2(E)&\mbox{if}\ b'c'\equiv 1\pmod 4.\end{cases}$$

If $p\equiv5\pmod8$,
then we have
$$N_2(\Q(\sqrt{-p}))=\{1,5,6,14\}\Q_2^{\times2}.$$ Note that $a'b'\equiv 5\pmod 8$,
by \cite[1.2]{EH78}, we can obtain the following result:
$$\theta(O^+(L_2))=\begin{cases}\Q_2^{\times}\ \text{or}\ \Z_2^{\times}\Q_2^{\times2}&\mbox{if}\ b'c'\not\equiv 2+(-1)^{\nu_2(b)}\pmod 4 ,\\\{\gamma\in\Q_2^{\times}: (\gamma,-5)_2=1\}= N_2(E)&\mbox{if}\ b'c'\equiv 2+(-1)^{\nu_2(b)}\pmod 4.\end{cases}$$
Hence, by the above results and \cite[Theorem 2(b)]{EHH}, when $\nu_2(a)=\nu_2(b)$,
then $\theta(O^+(L_2))\subseteq N_2(E)$ and $\theta^*(L_2,w)=N_2(E)$ if and only if (iii) of $(3)$ is satisfied.

In view of the above, we have $w=2^{\nu_2(c)}\mathcal{SF}(pa'b'c')$ is a primitive spinor exception of $\gen(L)$, and since $pa'b'\equiv 1\pmod {2^{3-\nu_2(c)}}$, we also have $w\equiv c\pmod8$. As in the proof
of Theorem \ref{Thm3.4}, one may easily verify that $f_{a,b,c,p^k}$ is not almost universal if (4) is
satisfied.

The proof of the converse is similar to the proof in Theorem \ref{Thm3.4}.

This completes our proof of Theorem \ref{Thm3.5}.
\qed
\medskip

\noindent{\it \bf Proof of Theorem 1.6}

Set $w=2^{\nu_2(c)}\mathcal{SF}(pa'b'c')$ and $E=\Q(\sqrt{-wdL})=\Q(\sqrt{-2p})$, by virtue of the proof
in Theorem \ref{Thm3.4}, then for each prime $q\ne2$, we have
$\theta(O^+(L_q))\subseteq N_q(E)$ and $\theta^*(L_q,t)=N_q(E)$ if and only if both $(1)$ and $(2)$ hold.

Now we consider the prime $q=2$, it is easy to see that
$$N_2(E)=\{1,\ 2p,\ 1+2p,\ 4+2p\}\Q_2^{\times2}.$$
If $(3)$ is satisfied,
set $L_2\cong 2^{\nu_2(c)}\langle c',2^rp^kb',2^sp^ka'\rangle$,
$U=\langle c',2^rp^kb'\rangle$ and $W=\langle 2^rp^kb',2^sp^ka'\rangle$,
where $r=3-\nu_2(c)+\nu_2(b)$ and $s=3-\nu_2(c)+\nu_2(a)$.
The formulae for $\theta(O^+(L_2))$ can be found in Theorem \ref{Thm3.1}--\ref{Thm3.3}.
We shall divide the remaining proof into the following three cases.

{\it Case} 1. $\nu_2(a)>\nu_2(b)\ge2$.

In the present case, note that $a'b'\equiv p\pmod 8$, we have
$$\theta(O^+(L_2))=Q(\mathcal{P}(U))Q(\mathcal{P}(W))\Q_2^{\times2},$$
and
$$Q(\mathcal{P}(U))\Q_2^{\times2}=c'\theta(O^+(\langle1,2^rp^kb'c'\rangle)),$$
and
$$Q(\mathcal{P}(W))\Q_2^{\times2}=2^rp^kb'\theta(O^+(\langle1,2^{s-r}a'b'\rangle)).$$
Hence, we may easily obtain the following results:
$$Q(\mathcal{P}(U))\Q_2^{\times2}=\{c',2^rp^kb'\}\Q_2^{\times2},$$
and
$$Q(\mathcal{P}(W))\Q_2^{\times2}=\begin{cases}2^rp^kb'N_2(E)&\mbox{if}\ s-r\in\{1,3\},\\2^rp^kb'\{1,2p\}\Q_2^{\times2}&\mbox{if}\ s-r\ge5.\end{cases}$$

$$\theta(O^+(L_2))=\begin{cases}2^rp^kb'c'N_2(E)\cup N_2(E)&\mbox{if}\ s-r\in\{1,3\},\\2^rp^kb'c'\{1,2p\}\Q_2^{\times2}\cup\{1,2p\}\Q_2^{\times2}&\mbox{if}\ s-r\ge5.\end{cases}$$
By \cite[Theorem 2(c)]{EHH}, one may easily verify that
$\theta(O^+(L_2))\subseteq N_2(E)$ and $\theta^*(L_2,w)=N_2(E)$ if and only if
$2^{1+\nu_2(b)}p^kb'c'\in N_2(E)$.

{\it Case} 2. $\nu_2(a)>\nu_2(b)=1$.

In this case, note that $a'b'\equiv p\pmod 8$ and $L_2\cong 2^{\nu_2(c)}\langle c',2^4p^kb',2^sp^ka'\rangle$.
If $s\in\{5,7\}$, then $L_2$ is of {\it Type} E, and hence $\theta(O^+(L_2))\not\subseteq N_2(E)$.

If $s\ge 9$, we have
$$\theta(O^+(L_2))=Q(\mathcal{P}(U))Q(\mathcal{P}(W))\Q_2^{\times2},$$
and
$$Q(\mathcal{P}(U))\Q_2^{\times2}=c'\theta(O^+(\langle1,2^4p^kb'c'\rangle)),$$
and
$$Q(\mathcal{P}(W))\Q_2^{\times2}=2^4p^kb'\theta(O^+(\langle1,2^{s-4}a'b'\rangle)).$$
Hence, we may easily obtain the following results:
$$Q(\mathcal{P}(U))\Q_2^{\times2}=\{c',5c'\}\Q_2^{\times2}\cup\{p^kb', 5p^kb'\}\Q_2^{\times2},$$
and
$$Q(\mathcal{P}(W))\Q_2^{\times2}=p^kb'\{1,2p\}\Q_2^{\times2}.$$
Thus, we have $5\in\theta(O^+(L_2))$ in the present case. However,
since $N_2(E)=\{1,\ 2p,\ 1+2p,\ 4+2p\}\Q_2^{\times2}$ with $1+2p\not\equiv1\pmod4$, hence $\theta(O^+(L_2)\not\subseteq N_2(E)$.

{\it Case} 3. $\nu_2(a)>\nu_2(b)=0$.

In the present case, if $\nu_2(c)=2$, then we have $L_2^{1/4}\cong\langle c', 2p^kb',2^2p^ka'\rangle$. By \cite[1.1]{EH78}, $L_2$ is of {\it Type} E and hence
$\theta(O^+(L_2))=\Q_2^{\times}\not\subseteq N_2(E)$. So, we must have $4\nmid c$.

If $\nu_2(c)=1$, then $L_2^{1/2}\cong\langle c', 2^2p^kb',2^sp^ka'\rangle$. If $s\in\{3,5\}$,
then $L_2$ is of {\it Type} E. So, we just need consider the case when $s>5$, then
$U=\langle c',2^2p^kb'\rangle$ and $W=\langle 2^2p^kb',2^sp^ka'\rangle$, we have
\begin{align*}
&Q(\mathcal{P}(U))\Q_2^{\times2}=c'\{\gamma\in\Z_2^{\times}\Q_2^{\times2}: (\gamma,-p^kb'c')_2=1\},
\\&Q(\mathcal{P}(W))\Q_2^{\times2}=p^kb'\Q_2^{\times2}\cup2p^ka'\Q_2^{\times2}.
\end{align*}
It is easy to see that
$$Q(\mathcal{P}(U))\Q_2^{\times2}=\begin{cases}\{c',5c'\}\Q_2^{\times2}&\mbox{if}\ p^kb'c'\equiv1\pmod4,\\\Z_2^{\times}\Q_2^{\times2}&\mbox{if}\ p^kb'c'\equiv3\pmod4.\end{cases}$$
Hence, we always have $5\in \theta(O^+(L_2)$. However,
since $N_2(E)=\{1,\ 2p,\ 1+2p,\ 4+2p\}\Q_2^{\times2}$ with $1+2p\not\equiv1\pmod4$, hence $\theta(O^+(L_2)\not\subseteq N_2(E)$.

If $2\nmid c$, then $L_2\cong\langle c', 2^3p^kb',2^sp^ka'\rangle$, if $s=4$, then $L_2$ is of {\it Type} E.
So, we just need consider the case when $s>4$.
In the present case, $U=\langle c',2^3p^kb'\rangle$ and $W=\langle 2^3p^kb',2^sp^ka'\rangle$, then $$\theta(O^+(L_2))=Q(\mathcal{P}(U))Q(\mathcal{P}(W))\Q_2^{\times2}.$$
Note that $a'b'\equiv p\pmod 8$ in the present case, then we have
$$Q(\mathcal{P}(U))\Q_2^{\times2}=c'\{\gamma\in\Q_2^{\times}: (\gamma,-2p^kb'c')_2=1\},$$
and
$$Q(\mathcal{P}(W))=\begin{cases}2p^kb'N_2(E)&\mbox{if}\ s=6,\\2p^kb'\{1,2p\}\Q_2^{\times2}&\mbox{if}\ s\ge8.\end{cases}$$
If $2p^kb'c'\not\in N_2(E)$, since $2p^kb'c'\in\theta(O^+(L_2))$, we have $\theta(O^+(L_2)\not\subseteq N_2(E)$.
If $2p^kb'c'\in N_2(E)$, then we have $2p^kb'c'\in \{2(p+2)\Q_2^{\times2},\ 2p\Q_2^{\times2}\}$.
Suppose first that $2p^kb'c'\in 2(p+2)\Q_2^{\times2}$, then we have
$$Q(\mathcal{P}(U))\Q_2^{\times2}=c'N_2(\Q(\sqrt{-2(p+2)}))=c'\{1,5+2p,2(p+2),2(p+4)\}\Q_2^{\times2},$$
since $5+2p\not\in (1+2p)\Q_2^{\times2}$, then one may easily verify that
$\theta(O^+(L_2)\not\subseteq N_2(E)$.
Suppose now that $2p^kb'c'\in 2p\Q_2^{\times2}$, then we have
$Q(\mathcal{P}(U))\Q_2^{\times2}=c'N_2(E)$, one may easily verify that
$\theta(O^+(L_2)\subseteq N_2(E)$. Hence, by the above results and \cite[Theorem 2(c)]{EHH},
in Case $3$, $\theta(O^+(L_2))\subseteq N_2(E)$ and $\theta^*(L_2,w)=N_2(E)$ if and only if
$2\nmid c$, $\nu_2(a)\ge 3$ and $p^kb'c'\equiv p\pmod 8$.

In view of the above, it is easy to see that $w=\mathcal{SF}(pa'b'c')$
is a primitive spinor exception of $\gen(L)$, and since $pa'b'\equiv 1\pmod 8$, we also have
$w\equiv c\pmod8$. By the proof
of Theorem \ref{Thm3.4}, one may easily verify that $f_{a,b,c,p^k}$ is not almost universal if (4) is
satisfied.

The proof of the converse is similar to the proof in Theorem \ref{Thm3.4}.

Combining the above we finally obtain the desired result.
\qed
\medskip

\subsection*{Acknowledgement}
I would like to thank my advisor, Zhi-Wei Sun, for his steadfast encouragement. I appreciate
the interest and feedback of Li Yang, Hao Pan, He-Xia Ni and Guo-Shuai Mao. Finally, I am exceedingly thankful for the careful reading and indispensable suggestions of the anonymous referee.


\begin{thebibliography}{99}

\bibitem{Be}  B. C. Berdnt, {\it Number Theory in the Spirit of Ramanujan}, Amer. Math. Soc., Prov-
idence, RI, 2006.
\bibitem{B}  F. van der Blij, {\it On the theory of quadratic forms}, Ann. Math {\bf 50} (1949), 875--883.
\bibitem{C} J. W. S. Cassels, {\it Rational Quadratic Forms}, Academic Press, London, 1978.
\bibitem{WKCOH} W. K. Chan and B.-K. Oh, {\it Almost universal ternary sums of triangular numbers},
    Proc. Amer. Math. Soc. {\bf 137} (2009), 3553--3562.
\bibitem{WKCANNA} W. K. Chan and A. Haensch, {\it Almost universal ternary sums of squares and triangular numbers}, Quadratic and Higher Degree Forms, Dev. Math, vol. 31, Springer-Verlag, New York, 2013, 51--62.
\bibitem{DR} W. Duke and R. Schulze-Pillot, {\it Representations of integers by positive ternary quadratic forms and equidistribution of lattice points on ellipsoid}, Invent. Math. {\bf99} (1990), 49--57.
\bibitem{EH75} A. G. Earnest, J. S. Hsia, {\it Spinor norms of local integral rotations II},
    Pacific J. Math. {\bf 61} (1975), no.1, 71--86.
\bibitem{EH78} A. G. Earnest, J. S. Hsia, {\it Spinor genera under field extensions II: 2 unramified in
    the bottom field}, Amer. J. Math. {\bf 100} (1978), no.3, 523--538.
\bibitem{EHH} A. G. Earnest, J. S. Hsia and D. C. Hung, {\it Primitive representations by spinor genera of ternary quadratic forms}, J. London Math. Soc., {\bf 50} (1994), 222--230.
\bibitem{GPS} S. Guo, H. Pan and Z.-W. Sun, {\it Mixed sums of squares and triangular numbers (II)},
Integers {\bf 7} (2007), \#A56, 5pp (electronic).
\bibitem{ANNA} A. Haensch, {\it A characterization of almost universal ternary quadratic polynomials with odd prime power conductor}, J. Number Theory {\bf 141} (2014), 202--213.
\bibitem{Lang} S. Lang. {\it Algebraic Numbers Theory}, Addison-Wesley, 1970.
\bibitem{KS} B. Kane and Z.-W. Sun, {\it On almost universal mixed sums of squares and triangular numbers},
¡¡¡¡Trans. Amer. Math. Soc. {\bf 362} (2010), 6425--6455.
\bibitem{Ki} Y. Kitaoka, {\it Arithmetic of Quadratic Forms}, Cambridge Tracts in Math., Vol. 106, 1993.
\bibitem{Ke56} M. Kneser, {\it Klassenzahlen indefiniter quadratischer Formen in drei oder mehr Veranderlichen}, Arch. Math. (Basel) {\bf7} (1956), 323--332.
\bibitem{Ke}  M. Kneser, {\it Darstellungsmasse indefiniter quadratischer Formen}, Math. Z. {\bf 11} (1961), 188--194.
\bibitem{Oto} O. T. O'Meara, {\it Introduction to Quadratic Forms}, Springer-Verlag, New York, 1963.
\bibitem{SP80} R. Schulze-Pillot,{\it Darstellung durch Spinorgeschlechter ternarer quadratischer Formen}, J. Number Theory {\bf 12} (1980), 529--540.
\bibitem{SP00} R. Schulze-Pillot, {\it Exceptional integers for genera of integral ternary postive definite quadratic forms}, Duke Math. J. {\bf 102} (2000) 351--357.
\bibitem{S07} Z.-W. Sun, {\it Mixed sums of squares and triangular numbers}, Acta Arith, {\bf 127} (2007),
    103--113.
\bibitem{S15} Z.-W. Sun, {\it On universal sums of polygonal numbers}, Sci. China Math, {\bf 58} (2015), 1367--1396.
\bibitem{S16} Z.-W. Sun, {\it A result similar to Lagrange¡¯s theorem}, J. Number Theory {\bf 162} (2016),
    190--211.
\bibitem{S17B} Z.-W. Sun, {\it On $x(ax+1)+y(by+1)+z(cz+1)$ and $x(ax+b)+y(ay+c)+z(az+d)$}, J. Number Theory
    {\bf171} (2017), 275--283.
\bibitem{S17} Z.-W. Sun, {\it On universal sums $x(ax+b)/2$+$x(cx+d)/2$+$x(ex+f)/2$}, preprint,
    {\tt arXiv:1502.03056v4}, 2017.
\bibitem{WS} H.-L. Wu and Z.-W. Sun, {\it On the 1-3-5 conjecture and related topics}, preprint,
    {\tt arXiv:1710.08763}, 2017.
\bibitem{X}  F. Xu, {\it Strong approximation for certain quadric fibrations with compact fibers}, Adv. Math. {\bf 281} (2015), 279--295.
\end{thebibliography}
\end{document}